\documentclass[12pt]{iopart}

\usepackage{color}
\usepackage{graphicx}
\usepackage{amscd}
\usepackage{amsmath}
\usepackage{amsfonts}
\usepackage{amssymb}
\usepackage[active]{srcltx}
\usepackage{theorem}
\theorembodyfont{\rmfamily}

\DeclareMathOperator{\sign}{sgn}
\DeclareMathOperator{\linspan}{span}
\DeclareMathOperator{\ex}{ex}
\DeclareMathOperator{\supp}{supp}
\DeclareMathOperator{\sym}{sym}
\DeclareMathOperator*{\argmin}{arg\,min}

\newcommand{\be}{\beta}

\newcommand{\N}{\mathbb N}

\newtheorem{assumption}{Assumption}
\newtheorem{remark}{Remark}
\newtheorem{lemma}{Lemma}
\newtheorem{theorem}{Theorem}
\newtheorem{proposition}{Proposition}

\theoremstyle{definition}

\newcommand{\clconv}{\overline{\mathrm{conv}}}
\newcommand{\NN}{\mathbf{N}}
\newcommand{\RR}{\mathbf{R}}
\newcommand{\lpspace}[1]{\ell^{#1}}
\newcommand{\ccspace}{c_0}
\newcommand{\abs}[2][]{|{#2}|_{#1}}
\newcommand{\norm}[2][]{\|{#2}\|_{#1}}
\newcommand{\Bignorm}[2][]{\Bigl\|{#2}\Bigr\|_{#1}}
\newcommand{\card}{\#}
\newcommand{\set}[2]{\{{#1}\,\bigl|\,{#2}\}}
\newcommand{\Bigset}[2]{\Bigl\{{#1}\,\Bigl|\,{#2}\Bigr\}}
\newcommand{\sett}[1]{\{{#1}\}}
\newcommand{\wstarrightarrow}{\rightharpoonup^*}
\newcommand{\wrightarrow}{\rightharpoonup}
\newcommand{\scp}[2]{\langle{#1},{#2}\rangle}
\newcommand{\without}{\backslash}
\newcommand{\placeholder}{\,\cdot\,}

\newcommand{\fd}{f^{\delta}}

\def\be#1{\begin{equation}\label{#1}}
\def\ee{\end{equation}}
\def\dup#1{\langle #1\rangle}
\def\norm#1{\hspace*{0.2ex}\|#1\|}
\newcommand{\ud}[1][]{u^{\dagger#1}}

\def\proof{\par{\bf Proof:} \ignorespaces}
\def\endproof{{\ \vbox{\hrule\hbox{%
   \vrule height1.3ex\hskip0.8ex\vrule}\hrule

    }}\par}

\newcommand{\mA}{\mathcal{A}}
\newcommand{\mH}{\mathcal{H}}

\newcommand{\exclude}[1]{}

\begin{document}

\title{The least error method for sparse solution reconstruction}
\author{K. Bredies${}^1$ and B. Kaltenbacher${}^2$ and E. Resmerita${}^2$}
\address{${}^1$ Institute of Mathematics and Scientific Computing, Karl-Franzens-Universit\"at Graz, Austria,\\
${}^2$ Institute of Mathematics, Alpen-Adria-Universit\"at Klagenfurt, Austria}
\ead{kristian.bredies@uni-graz.at and barbara.kaltenbacher@aau.at and elena.resmerita@aau.at}
\begin{abstract}
This work deals with a regularization method enforcing solution sparsity of linear  ill-posed problems by appropriate discretization in the image space. Namely, we formulate the so called  least error method in an $\ell^1$ setting and perform the convergence analysis by choosing the discretization level according to  an a priori rule, as well as two a posteriori rules, via the discrepancy principle and the monotone error rule, respectively.
Depending on the setting, linear or sublinear convergence rates  in the $\ell^1$-norm are obtained under a source condition yielding sparsity of the solution. A part of the study is devoted to analyzing the structure of the approximate solutions and of the involved source elements.

\end{abstract}

\maketitle

\section{Introduction}
In order to recover sparse solutions of linear operator equations, it is common to consider Tikhonov regularization with $\ell^1$-penalty (see, e.g. \cite{DDD04}). In this paper, we focus on a different regularization method based on discretization known in the literature as  the  least error or the dual least squares method, and taking advantage of the $\ell^1$ framework. The reader is referred e.g., to \cite{Natt77}, \cite{VH85}, \cite{Engl96} for some classical analysis of the above method in Hilbert spaces and to \cite{HKKR15}, for recent results in some classes of Banach spaces. Thus, it has been shown in \cite{HKKR15} that the least error method converges in spaces with good smoothness and convexity properties, which is not the case in the considered sparsity context. Thus, to the best knowledge of the authors, this is the first time the least error approach is analyzed in the context of sparse regularization in $\ell^1$. Technically speaking, this analysis differs essentially as regards stability estimates and convergence of the method for an a priori rule, which is the backbone of convergence for the method combined with the monotone error rule or the discrepancy principle for choosing the discretization level, playing the role of a regularization parameter here.
Under a source condition we get a convergence rate result, not only for the Bregman distance, but even for the full $\ell^1$ norm. This is a consequence of the sparsity structure of the exact solution induced by the source condition, that enables a special stability estimate and an ideal error rate $O(\delta)$ as the noise level $\delta$ tends to zero, under certain a priori information, similar to \cite{GHO08} and \cite{BL09, G09} for the case of non-convex sparse regularization. A convergence rate with a posteriori choice of the discretization level can be alternatively obtained with the discrepancy principle.

Throughout the paper, let $H$ be a Hilbert space and $A: H \to \ccspace$ be linear and continuous.
Then, with the identification $\ccspace^* = \lpspace{1}$,
the mapping $A^*: \lpspace{1} \to H$ is weak*-to-weak
continuous in addition
to being linear and continuous. For the time being, we do not assume that
$A$ is injective, but will make this assumption later, observing already that
$A^*$ will also be injective in this case.

We would like to solve the inverse problem
\begin{equation}\label{equation}
A^* u = f
\end{equation}
provided only data $f^\delta$ satisfying 
\begin{equation}\label{delta}
\norm{f^\delta - f}_H \leq \delta\,,
\end{equation}
an assumption that is also made throughout the paper.

The aim of this study is to solve the equation by discretization in the image space $H$. That is, choose a sequence of subspaces $(H_n)_n$ of $H$ where each
$H_n$ is finite-dimensional with dimension $n$ and
 $\lim_{n \to \infty} P_n v = v$ for each $v \in H$, with $P_n$ denoting
the orthogonal projection onto $H_n$, i.e., 
\begin{equation}\label{Pn}
P_n=\mbox{Proj}_{H_n}.
\end{equation}
The least error method defines
\be{le}
u^n\in \argmin\{ \norm{u}_{1} \, : \, \forall z^n\in H_n \, : \ \dup{z^n,A^*u}=\dup{z^n,\fd}\},
\ee
which is equivalent to $u^n$ solving
\begin{equation}
  \label{eq:minprob}
  \min_{u \in \lpspace{1}} \ \norm{u}_{1} \qquad
  \text{subject to} \qquad
  P_nA^* u = P_n f^\delta.
\end{equation}
The structure of this work is as follows.  Well-definedness, an equivalent formulation of the least error method and a few useful estimates are shown in Section 2. A convergence analysis for an a priori choice of the discretization level, as well as for two a posteriori choices is provided in Section 3, 4, respectively. Convergence rates up to $O(\delta)$ are derived in Section 5, where  the specific structure of the (approximate) solutions and the corresponding source elements are also discussed.  Section 6 shortly reviews some particular instances of the least error method in the current setting. 

\section{The least error method in $\ell^1$}

We will use the identification of $\ell^1$ with the dual  of the space $c_0$ of sequences converging to zero  and the weak$^*$ compactness of the sublevel sets of the $\ell^1$ norm. Recall  that 
\be{subgr}
\partial(\norm{\cdot}_1)(u)=\{\xi\in\ell^\infty: 
\norm{\xi}_\infty \leq 1 \quad
\text{and} \quad 
\dup{\xi,u}=\norm{u}_1\}
\ee
due to convexity and homogeneity of the $\ell^1$-norm. More specifically, by
exploiting the  structure of the $\lpspace{1}$-norm, we have
\be{sign}
\partial(\norm{\cdot}_1)(u) = 
\sign(u) = \set{\xi \in \lpspace{\infty}}{\xi_i = \tfrac{u_i}{\abs{u_i}}
  \ \text{if} \ u_i \neq 0, \ \xi_i \in [-1,1] \ \text{if} \ u_i = 0}.
\ee

\bigskip

\noindent Problem \eqref{le} is well-defined, as stated below.

\begin{proposition}\label{welldef}
Assume that 
\be{kernel}
\mathcal{N}(A)\cap H_n =\{0\}\,.
\ee
Then the set of minimizers 
$\argmin\{ \norm{u}_1 \, : \, \forall z^n\in H_n \, : \ \dup{z^n,A^*u}=\dup{z^n,\fd}\}$ is nonempty. 
\end{proposition}

\begin{proof}
The proof is similar to the one in \cite{HKKR15}, showing that the feasible set is nonempty. Since \eqref{kernel} implies $\mathcal{N}(AP_n)=\{0\}$ and the range of the operator $P_nA^*$ is finite dimensional, hence closed, we can conclude
\[
\mathcal{R}(P_nA^*)=\mathcal{R}((AP_n)^*)=\mathcal{N}(AP_n)^\bot=H_n\,.
\]
 Weak$^*$-weak sequential continuity (which is ensured here, as mentioned above) of the operator $A^*$   implies weak$^*$-closedness of the feasible set. This together with coercivity of the objective function and  weak$^*$ compactness of the sublevel sets of the $\ell^1$ norm yield the result.
\end{proof}\medskip

Note that  weak$^*$-weak sequential continuity of the operator governing the equation to be solved has been considered in \cite{BFH13}, in the context of Tikhonov type regularization. \smallskip

\begin{proposition}\label{equiv} Let the assumptions of Proposition \ref{welldef}  be satisfied. Then \eqref{le} is equivalent to
\[
u^n\in  E_n \mbox{ and } \forall z^n\in H_n \, : \ \dup{z^n,A^*u^n}=\dup{z^n,\fd},   
\]
where 
\begin{equation}
  \label{eq:en_sets}
  E_n=(\partial \|\cdot\|_1)^{-1}(A H_n ).  
\end{equation}
\end{proposition}

\begin{proof}
Let $H_n= \linspan\{e_1,...,e_n\}$, i.e., 
the elements $e_i$ form a basis of $H_n$. Then problem \eqref{le} is equivalent to 
\be{le1}
u^n\in\mbox{argmin}\{ \norm{u}_1 \, : \, G(u)=0\},
\ee
where $G:\ell^1\to {\mathbb{R}}^n$ is given by $G(u)=Tu+b$ with $Tu=(\langle A^*u,e_i\rangle)_i$ and $b=(-\langle\fd,e_i\rangle)_i$. Since the function $\phi=\|\cdot\|_{1}$  is continuous and the finite dimensional rank operator $T$ has a closed range, one can apply Th.~3.20 in \cite{BP12} and obtain that $u^n$ is a solution of \eqref{le1} if and only if
$$\partial \phi(u^n)\cap \mathcal{R}(T^*)\neq\emptyset.$$ 
As for any $u\in\ell^1$, $p\in\mathbb{R}^n$
\[
(Tu)^T p =\sum_{i=1}^n \dup{A^*u,e_i}p_i=\dup{u,A \sum_{i=1}^n p_ie_i}
\]
and therefore 
$\mathcal{R}(T^*)=A(\linspan\{e_1,...,e_n\})=AH_n$, 
the proof is complete.
\end{proof}

\begin{remark}
For interpreting the optimality conditions derived in Proposition~\ref{equiv}, 
we recall that
\[
(\partial\norm{\cdot}_1)^{-1} 
= 
\partial(I_{\sett{\norm{\xi}_\infty \leq 1}}),
\]
where the subgradient has to be understood as
subset of the predual space $\lpspace{1}$ and reads as
\[
\begin{aligned}
  \norm{\xi}_\infty > 1: & & 
  \partial(I_{\sett{\norm{\xi}_\infty \leq 1}})(\xi) 
  &= \emptyset, \\
  \norm{\xi}_\infty \leq 1: & & 
  \partial(I_{\sett{\norm{\xi}_\infty \leq 1}})(\xi) 
  &=
  \{u \in \lpspace{1} \, | \, u_i = 0 \ \text{if} \ \xi_i \in {]{-1,1}[}, \\
  & & & 
  \qquad
  u_i \geq 0 \ \text{if} \ \xi_i = 1 \ \text{and} \ 
  u_i \geq 0 \ \text{if} \ \xi_i = -1\}.
\end{aligned}
\]
Thus, according to Proposition~\ref{equiv}, $u^n$ solves \eqref{le} iff there exists an element $v^n\in H_n$ such that
\[
\norm{Av^n}_\infty\leq 1 \mbox{ and } (u^n)_i\begin{cases} 
=0 & \mbox{ if }(Av^n)_i\in]-1,1[\\
\geq0 & \mbox{ if }(Av^n)_i=1\\
\leq0 & \mbox{ if }(Av^n)_i=-1
\end{cases}
\]
\end{remark}

\medskip
\noindent
In order to show stability of  the discretization method, we define 
\be{kappa}
\kappa_n=\sup_{z^n\in H_n}\frac{\norm{z^n}}{\norm{Az^n}_\infty}=\frac{1}{\inf_{\norm{z^n}=1}\norm{Az^n}_\infty}.
\ee
Note that these values are finite due to \eqref{kernel}. 

The Bregman distance with respect to the $\ell^1$ norm and an element $\xi_v\in \partial \|\cdot\|_1(v)$ is defined by
\[
D(u,v)=\norm{u}_1-\norm{v}_1-\dup{\xi_v,u-v}=\norm{u}_1-\dup{\xi_v,u},
\]
see~\eqref{subgr},
the symmetric Bregman distance by
 \be{symmBD}
D^{\sym}(v,u)=D(v,u)+D(u,v)=\dup{\xi_v-\xi_u,v-u},
\ee
with $\xi_u\in \partial \|\cdot\|_1(u)$. 

\begin{lemma} (compare \cite[Lemma 4.5]{HKKR15})
Under the conditions of Proposition~\ref{welldef},
\be{in}
\norm{u^n}_1\leq \delta \kappa_n+\norm{\ud}_1\,
\ee
holds.
\end{lemma}
\proof
Due to Proposition \ref{equiv}, each solution $u^n$ of \eqref{le} satisfies $\xi^n=Av^n\in \partial \|\cdot\|_1(u^n)$ for some $v^n\in H_n$. Thus,
\begin{align}
\notag
\norm{u^n}_1=\langle\xi^n,u^n\rangle &=\langle Av^n,u^n\rangle\\
\notag
&=\langle v^n,A^{*}u^n\rangle=\langle v^n,\fd\rangle\\
\notag
&=\langle v^n,\fd-f\rangle+\langle Av^n,\ud\rangle\\
\notag
&\leq\delta \kappa_n\|Av^n\|_{\infty}+\norm{\ud}_1\|Av^n\|_{\infty}\\
&\leq \delta \kappa_n+\norm{\ud}_1 \tag*{\endproof}
\end{align}

\begin{proposition}\label{prop:stabest}
 Let the assumptions of Proposition \ref{welldef} be satisfied  and  let
\[
u^{n,i}\in\mbox{argmin}\{ \norm{u}_1 \, : \, \forall z^n\in H_n \, : \ \dup{z^n,A^*u}=\dup{z^n,f^i}\} \quad \text{for} \quad i=1,2\,,
\]
be well-defined solutions of \eqref{eq:minprob} corresponding to data $f^1, f^2\in H$, respectively. 

Then the following estimates hold:
\be{noisy_estim}
D^{\sym}(u^{n,1},u^{n,2})\leq 2\kappa_n \norm{f^1-f^2},
\ee
\be{kk}
|\norm{u^{n,1}}_1-\norm{u^{n,2}}_1|\leq 2\kappa_n \norm{f^1-f^2}.
\ee
\end{proposition}

\begin{proof} Let $v^{n,i}\in H_n$ such that $\xi^{n,i}=Av^{n,i}\in \partial (\norm{\cdot}_1)(u^{n,i})$, for $i=1,2$. Then \eqref{noisy_estim} follows from
\begin{eqnarray*}
D^{\sym}(u^{n,1},u^{n,2})&=&\dup{Av^{n,1}-Av^{n,2}, u^{n,1}-u^{n,2}}\\
&=& \dup{v^{n,1}-v^{n,2}, A^*u^{n,1}-A^*u^{n,2}}\\
&=&  \dup{v^{n,1}-v^{n,2}, f^1-f^2}\\
&\leq& \kappa_n\norm{Av^{n,1}-Av^{n,2}}_\infty\norm{f^1-f^2}\\
&\leq& 2\kappa_n\norm{f^1-f^2},
\end{eqnarray*}
where the last inequality is a consequence of \eqref{subgr}.

In order to show \eqref{kk}, let $u\in \ell^1$ be a solution of \eqref{le} with $f^1-f^2$ instead of $\fd$. According to \eqref{noisy_estim},
plugging in $u$ and $0$ instead of $u^{n,1}$ and $u^{n,2}$ as well as
$f^1 - f^2$ and $0$ instead of $f^1$ and $f^2$, one has
\begin{equation*}
\norm{u}_1=D^{\sym}(u,0)\leq  2\kappa_n\norm{f^1-f^2}.
\end{equation*}
Since $u^{n,1}+u$ satisfies $\dup{z^n,A^*(u^{n,1}+u)}=\dup{z^n,f^2}$, it follows that
\[
\norm{u^{n,2}}_1\leq \norm{u^{n,1}+u}_1\leq\norm{u^{n,1}}_1+\norm{u}_1\leq\norm{u^{n,1}}_1+ 2\kappa_n\norm{f^1-f^2},
\]
which, by symmetry, implies \eqref{kk}.
\end{proof}

\section{Convergence with a priori choice of $n$}

We state below a convergence result in case of an a priori choice of the discretization dimension $n$.
We will use the following notation for solutions in the exact data case:
\begin{equation}\label{exact_sol}
\ud[,n]\in\mbox{argmin}\{ \norm{u}_{1} \, : \, \forall z^n\in H_n \, : \ \dup{z^n,A^*u}=\dup{z^n,f}\}.
\end{equation}
Note that for the following convergence result,
instead of pointwise convergence of the projections $P_n$, 
only a weaker condition is needed for proving convergence 
with a priori choice of $n$.

\begin{theorem}\label{apriori_le}
Let the assumptions of Proposition \ref{welldef} be satisfied and assume that \eqref{equation} is solvable.  Additionally, assume that  
\be{adjoint_conv}
\forall z\in H\, : \ \inf_{z^n\in H_n} \norm{A(z-z^n)}\to0\mbox{ as }n\to\infty\,.
\ee
Then the following statements hold:
\begin{enumerate}
\item[(a)] For exact data $\delta=0$ one has convergence 
\[\norm{\ud[,l]-\ud}_1\to 0 \mbox{ as }l\to\infty\,,\]
where $(\ud[,l])_l$ is a subsequence of $(\ud[,n])_n$ with terms given by \eqref{exact_sol} and $\ud$ is a solution of \eqref{equation}. 
\item[(b)] Let the noisy data $\fd$ satisfy \eqref{delta},  the dimension $n=n_{AP}(\delta)$ be chosen such that
\be{apriorirule}
n_{AP}(\delta)\to\infty\mbox{ and }\delta\kappa_{n_{AP}(\delta)}\to0 \mbox{ as }\delta\to0\,
\ee
and the sequence $(\delta_m)_m$ in $(0,+\infty)$ converge to zero. Then there exists a subsequence $(\delta_l)_l$ such that
\be{conv_apriori}
\lim_{l\to \infty}\norm{u^l-\ud}_1=0,
\ee
with $u^l:=u^{n_{AP}(\delta_l)}$ and $\ud$  a solution of \eqref{equation}.
\end{enumerate}
\end{theorem}

\begin{remark}
\end{remark}

\begin{proof} (a) Let $\ud$ be a solution of \eqref{equation}. Due to \eqref{exact_sol}, one has
\be{ineq1}
\norm{\ud[,n]}_1\leq \norm{\ud}_1.
\ee
Hence, the sequence $(\ud[,n])_{n}$ has a   weakly$^*$ convergent subsequence $(\ud[,l])_{l}$ with limit point $\tilde u$. Weak$^*$-weak continuity of the operator $A^*$ guarantees weak convergence of  $(A^*\ud[,n_{l}])_{l}$ to $A^*\tilde u$.   Moreover, equality $\dup{z^{l},A^*\ud[,l]-f}=0$ for all $z^{l}\in H_{l}$ and \eqref{adjoint_conv} imply
\[
\begin{aligned}
\forall z\in H\, : \
\dup{z,A^*\ud[,l]-f}=&\inf_{z^{l}\in H_{l}}\dup{z-z^{l},A^*\ud[,l]-f}\\
=&\inf_{z^{l}\in H_{l}}\dup{z-z^{l},A^*(\ud[,l]-\ud)}\\
\leq&2\inf_{z^{l}\in H_{l}}\norm{A(z-z^{l})}  \norm{\ud} \to 0 \mbox{ as }l\to\infty\,.
\end{aligned}
\]
Consequently, $(A^*\ud[,l])_{l}$ converges weakly also to $f$, which means that $f$ must equal $A^*\tilde u$. Now  \eqref{exact_sol} and weak$^*$ lower semicontinuity of the $\ell^1$ norm imply, together with~\eqref{ineq1},
\[
\norm{\tilde u}_1\leq\liminf_{l\to\infty}\norm{\ud[,l]}_1\leq \limsup_{l\to\infty}\norm{\ud[,l]}_1\leq \norm{\tilde u}_1,
\]
that is $\lim_{l\to \infty}\norm{\ud[,l]}_1=\norm{\tilde u}_1.$
From this and weak$^*$ convergence of $(\ud[,l])_{l}$ to $\tilde u$ one deduces 
\be{norm_conv}
\lim_{l\to \infty}\norm{\ud[,l]-\tilde u}=0,
\ee
based on the Kadec-Klee property in $\ell^1$ (see, e.g. \cite{BFH13}). 

(b) Denote $n_m:=n_{AP}(\delta_m)$ and let $u^m$ be  a solution of \eqref{le} corresponding to the subspace $H_{n_m}$ and to the noisy data $f^{\delta_m}$. Due to \eqref{in}, \eqref{kk}, and \eqref{apriorirule} 
one obtains boundedness of the sequence $(u^m)_m$. By using the  proof idea of a), existence of a subsequence $(u^l)_l$ follows, such that its strong limit point $\tilde u$ is a solution of \eqref{equation}.
\end{proof}


\section{Convergence with a posteriori choice of $n$
}

Convergence with respect to the a posteriori monotone error rule follows in a manner similar to the one for `nice' spaces --- see \cite{HKKR15}, with some differences due to the space setting. For the sake of completeness, we formulate and prove the result below.

\begin{theorem}\label{apost_le}
Let the assumptions of Theorem \ref{apriori_le}  be satisfied and let $\ud$ be a solution of \eqref{equation}.  Then one has
\begin{enumerate}
\item[(a)] There exists $v^n\in H_n $ such that $u^n\in (\partial\norm{\cdot}_1)^{-1}(Av^n)$.
\item[(b)]  The identity $\norm{u^n}_1=\dup{v^n,\fd}$ holds, where  $v^n$ is chosen as in (a). If 
\be{HnHnp}
H_n\subseteq H_{n+1}\,,
\ee 
then
\[
\norm{u^n}_1\leq \norm{u^{n+1}}_1.
\]
\item[(c)] Let $d_{ME}(n)$ stand for
\[
d_{ME}(n)= \frac{\dup{v^{n+1}-v^n,\fd}}{\norm{v^{n+1}-v^n}} \quad
\text{for} \quad v^{n+1} \neq v^n \qquad
\text{and} \qquad d_{ME}(n) = 0 \quad \text{else},
\]
then the following hold:
\[
D(\ud,u^{n+1})-D(\ud,u^n) \leq -(d_{ME}(n)-\delta) \norm{v^{n+1}-v^n}\,.
\]
and, in case $v^{n+1} \neq v^n$,
\[
d_{ME}(n)= \frac{\norm{u^{n+1}}_1-\norm{u^{n}}_1}{\norm{v^{n+1}-v^n}}
\]
where the Bregman distances are with respect to $\xi^{n+1} = A v^{n+1}$ and
$\xi^n = A v^{n}$, respectively.
Additionally, if \eqref{HnHnp} holds, then  
\[
d_{ME}(n)=\frac{D(u^{n+1},u^n)}{\norm{v^{n+1}-v^n}}\geq0,
\]
and the error measured in the Bregman distance is monotonically decreasing as long as 
\be{deltadME}
\delta \leq d_{ME}(n)\,.
\ee
\item[(d)] Let \eqref{HnHnp} hold for all $n\in\mathbb{N}$ and let $n=n_{ME}(\delta)$ be the first index such that \eqref{deltadME} is violated
\begin{equation}\label{monerr}
n_{ME}(\delta)=\min\{n\in\N\ : \ v^{n+1} \neq v^{n} \ \text{and} \ 
\frac{D(u^{n+1},u^n)}{\norm{v^{n+1}-v^n}}<\delta\}.
\end{equation}

If $n_{ME}(\delta)\to\infty$ as $\delta\to0$ and \eqref{adjoint_conv} holds, then $D(\ud,u^{n_{ME}(\delta)})\to0$ as $\delta\to0$ subsequentially.
\end{enumerate}
\end{theorem}

\begin{proof}
Item (a) has already been proven by Proposition \ref{equiv}.\\
Due to (a) and  \eqref{subgr},
we get the first part of item (b) by virtue of
\[
\norm{u^n}_1 
=\dup{Av^n,u^n}\\
=\dup{v^n,A^*u^n}=\dup{v^n,\fd}. 
\]
Due to  assumption \eqref{HnHnp}, the feasible set for $u^n$ contains the feasible set for $u^{n+1}$, hence 
\eqref{le} yields the second part of (b).

Note that
\begin{eqnarray*}
D(\ud,u^{n+1})-D(\ud,u^n) &=&
 \dup{\xi^n-\xi^{n+1},\ud}\\
&=&-\dup{A(v^{n+1}-v^n),\ud}\\
&=&-\dup{v^{n+1}-v^n,\fd}+\dup{v^{n+1}-v^n,\fd-f}\\
&\leq&-\dup{v^{n+1}-v^n,\fd}+\norm{v^{n+1}-v^n}\delta\\
&=&-(d_{ME}(n)-\delta)\norm{v^{n+1}-v^n}\,.
\end{eqnarray*}

The first identity for $d_{ME}(n)$ in (c) is an immediate consequence of (b), while the second one follows in case of \eqref{HnHnp} from $\dup{v^n,A^*u^{n+1}}=\dup{v^n,A^*u^{n}}$ which can be rewritten as $\dup{\xi^n,u^{n+1}-u^n}=0$.

For showing item (d), let $n_{AP}(\delta)$ be an a priori stopping rule satisfying \eqref{apriorirule}, let $(\delta_k)_{k}$ be a sequence of noise levels tending to zero and denote by $n_{AP}^k=n_{AP}(\delta_k)$, $n_{ME}^k=n_{ME}(\delta_k)$ the stopping indices chosen by the a priori and the monotone error rule, respectively.\\
If there exists $k_0$ such that $n_{ME}^k>n_{AP}^k$ for all $k\geq k_0$, then by monotone decay of the error up to $n_{ME}^k$ we have
$D(\ud,u^{n_{ME}^k})\leq D(\ud,u^{n_{AP}^k})\to0$ as $k\to\infty$ (actually one has strong convergence of the sequence $(u^{n_{AP}^k})_k$ to a solution).
Otherwise there exists a subsequence $(k_l)_{l}$ such that for all $l\in\N$ we have $n_{ME}^{k_l}\leq n_{AP}^{k_l}$ and therefore, by \eqref{HnHnp}, $\kappa_{n_{ME}^{k_l}}\leq \kappa_{n_{AP}^{k_l}}$, so the right hand limit in \eqref{apriorirule}
together with the assumption $n_{ME}(\delta) \to \infty$ 
implies strong convergence of $(u^{n_{ME}^{k_l}})_l$ to a solution of the 
equation.
\end{proof}

\medskip 

Besides the monotone error rule, which gives unconditional convergence, we also consider the discrepancy principle
\begin{equation}\label{discrprinc}
n_{DP}(\delta)=\min\{n\in\N\ : \ \norm{A^*u^n-f^\delta}\leq\tau\delta\}
\end{equation}
with some fixed $\tau>1$,
for which, as usual (cf., e.g., conditions (2.13), (2.14) in \cite{HKKR15}) certain assumptions on the discretization have to be made to guarantee well-definition and convergence.
We assume existence of constants $C_1,C_2$ such that for all $n\in\N$
\begin{equation}\label{C1}
\kappa_n \gamma_n\leq C_1
\end{equation}
\begin{equation}\label{C2}
\kappa_n\hat{\gamma}_{n-1}\leq C_2
\end{equation}
where 
\begin{equation}\label{gamman}
\gamma_n=\sup_{{u^1,u^2\in E_n}\atop{D^{\sym}(u^1,u^2)\not=0}} \frac{\norm{(\mbox{id}-P_n)A^*(u^1-u^2)}}{D^{\sym}(u^1,u^2)}\,, \quad
\hat{\gamma}_n=\norm{(\mbox{id}-P_n)A^*(\ud[,n]-u^\dagger)}
\end{equation}
with $\ud[,n]$ as in \eqref{exact_sol}, $P_n$ as in \eqref{Pn} and
$E_n$ as in~\eqref{eq:en_sets}.
\begin{theorem}
Let the assumptions of Proposition \ref{welldef} be satisfied, assume that \eqref{equation} is solvable and that the noisy data $\fd$ satisfy \eqref{delta}.  Additionally, assume that condition \eqref{C1} with $\tau>2C_1+1$ holds and that $\hat{\gamma}_n\to0$ as $n\to\infty$. Then $n_{DP}(\delta)$ according to the discrepancy principle \eqref{discrprinc} is well-defined. If additionally \eqref{C2} holds 
and the sequence $(\delta_m)_m$ in $(0,+\infty)$ converges to zero, then there exists a subsequence $(\delta_l)_l$ such that
\be{conv_dp}
u^l\stackrel{*}{\rightharpoonup}\ud \mbox{ as }{l\to \infty} \mbox{ in }\ell^1,
\ee
with $u^l:=u^{n_{DP}(\delta_l)}$ and $\ud$  a solution of \eqref{equation}. 
\end{theorem}
\begin{proof}
Using \eqref{eq:minprob}, \eqref{noisy_estim}, \eqref{C1}, \eqref{gamman} we get
\begin{align}
\notag
\norm{A^*u^n-f^\delta}
&=\Bignorm{(\mbox{id}-P_n)\Bigl(A^*(u^n-\ud[,n])+A^*(\ud[,n]-u^\dagger)+(f-f^\delta)\Bigr)}\\
\label{estdp}
&\leq (2C_1+1)\delta+\hat{\gamma}_n\,,
\end{align}
where $\hat{\gamma}_n$ tends to zero as $n\to\infty$, hence the right hand side is smaller than $\tau\delta$ for sufficiently large $n$. Consequently, $n_{DP}(\delta)$ is well-defined.
\\
On the other hand, \eqref{estdp} together with minimality in \eqref{discrprinc} yields 
\be{estNDPm1}
(\tau-2C_1-1)\delta < \hat{\gamma}_{n_{DP}(\delta)-1}\,
\ee
hence by \eqref{in} and \eqref{C2} we have 
\[
\norm{u^{n_{DP}(\delta)}}_1\leq \delta \kappa_{n_{DP}(\delta)}+\norm{\ud}_1 \leq \frac{C_2}{\tau-2C_1-1}+\norm{\ud}_1
\]
which yields uniform boundedness of $(\norm{u^{n_{DP}(\delta_m)}}_1)_m$, hence, as in the proof of Theorem \ref{apriori_le}, weak* subsequential convergence.
\end{proof}

\section{Convergence rates under a source condition}

We  assume throughout this section that $A$ and consequently, $A^*$ is
injective and that the following source condition
is satisfied:

\begin{assumption}\label{ass:sc}
There exists a source element $v^\dagger \in H$ such that 
$\norm{Av^\dagger}_\infty \leq 1$ and $(Av^\dagger)_i 
= \sign(u^\dagger_i)$ whenever $u_i^\dagger \neq 0$. 
\end{assumption}

Note that 
$Av^\dagger \in \ccspace$, hence there are only finitely many
$i$ for which $(Av^\dagger)_i \in \sett{-1,1}$ and, consequently,
only finitely many $i$ with $u^\dagger_i \neq 0$. The latter means that
the solution $u^\dagger$ has to be sparse.  

\subsection{The structure of the source element}
We first see that $v^\dagger$ can be assumed, without loss of 
generality, to satisfy 
$\abs{(Av^\dagger)_i} = 1$ if and only if $u^\dagger_i \neq 0$.

\begin{lemma}
  \label{lem:strict_source_elem}
  There exists a $v^\ddagger \in H$ with $\norm{A v^\ddagger}_\infty \leq 1$,
  $(Av^\ddagger)_i = \sign(u^\dagger_i)$ whenever $u^\dagger_i \neq 0$ and
  $\abs{(A v^\ddagger)_i} < 1$ whenever $u^\dagger_i = 0$.
\end{lemma}

\begin{proof}
  First, denote 
  \be{defI}
  I = \set{i \in \NN}{(Av^\dagger)_i \in \sett{-1,1}}
  \ee
  which is a finite set by $Av^\dagger\in\ccspace$, and  
  \be{defHI}
	H_I = \linspan\set{A^*e_i}{i \in I}
  \ee
  where $(e_i)_i$ is the canonical basis in $\ell^1$.
  By injectivity of $A^*$, 
  $\set{A^*e_i}{i \in I}$ are linearly independent.
  Thus, denoting 
  \be{defAI}
	A_I: H \to \RR^I\,, \quad (A_Iv)_i = (Av)_i \mbox{ for }i \in I,
  \ee
  we see that the mapping $A_IA_I^*: \RR^I \to \RR^I$
  is (continuously) invertible.
  Thus, the problem of finding a $v_I \in H_I$ such that 
  $\scp{v_I}{A^*e_i} = \bar v_i$ for $i \in I$ and given 
  $\bar v\in\RR^I$, is uniquely solvable with
  $\norm{v_I}_H \leq C \norm{\bar v}_\infty$ for some $C > 0$
  independent of $\bar v$.
  In particular, choosing $I_0 = \set{i \in I}{u^\dagger_i = 0}$ and,
  for $\varepsilon > 0$,
  \[
  \bar v_i =
  \begin{cases}
    -\varepsilon \sign(A v^\dagger)_i & \ \text{if} \ i \in I_0, \\
    0 & \ \text{if} \ i \in I\without I_0,
  \end{cases}
  \]
  we have, for the corresponding
  $v_I$, that $\norm{Av_I}_\infty \leq C\norm{A} \varepsilon$.
  Next, we know, again since $Av^\dagger \in \ccspace$, that 
  \be{defrho}
  \rho = 1 - \max_{i \in \NN \without I} \ \abs{(Av^\dagger)_i} > 0.
  \ee
Namely, for all $i \in \NN \without I$, we have $\abs{(Av^\dagger)_i}<1$, and assuming $\abs{(Av^\dagger)_{i_k}}\to1$ for some subsequence contradicts $Av^\dagger\in \ccspace$.
  Choosing $0 < \varepsilon < \min(1, \frac12 (C \norm{A})^{-1} \rho)$ 
  and letting
  $v^\ddagger = v^\dagger + v_I$ yield
  \[
  \left\{
  \begin{aligned}
    (A v^\ddagger)_i &= \sign(u^\dagger_i) & \ \ \text{for} \ \
    i \in I\without I_0, \\
    \abs{(A v^\ddagger)_i} &= 1 - \varepsilon < 1
    &\ \ \text{for} \ \ i \in I_0, \\
    \abs{(A v^\ddagger)_i} &\leq \abs{(A v^\dagger)_i} 
    + \norm{Av_I}_\infty \leq 1 - \rho + C \norm{A} \varepsilon
    \leq 1 - \frac{\rho}2 < 1 & \ \ \text{for} \ \ i \in 
    \NN \without I.
  \end{aligned}
  \right.
  \]
  This, however, immediately implies that $v^\ddagger$ possesses the
  stated properties.
\end{proof}

In the following assume that $v^\dagger$ is a source element which
satisfies $\abs{(Av^\dagger)_i} = 1$ if and only if $i \in \supp u^\dagger$.
Thus, from now on
\be{Isupp}
I=\supp u^\dagger\,.
\ee
We denote by 
\be{defeps}
\varepsilon_{v^\dagger} = 1 
- \max_{i \notin \supp u^\dagger} \abs{(Av^\dagger)_i}
\ee
which is a positive number since $Av^\dagger\in\ccspace$ (see the argument after \eqref{defrho}).

\bigskip
In the sequel, we will also make use of the projections operators $P_n$ as in \eqref{Pn}, in particular, the pointwise convergence $\lim_{n \to \infty} P_n v = v$
for each $v \in H$.
Then, for $n$ large enough, $u^\dagger$ is already a solution 
of~\eqref{eq:minprob} for exact data, i.e., $f^\delta =  f$.

\begin{lemma}
  \label{lem:discrete_solution_exact_data}
  There is an $n_0 \in \NN$ such that for $n \geq n_0$, $u^\dagger$ is 
  the unique 
  solution to~\eqref{eq:minprob} with $f^\delta =  f$.

  Furthermore, for $n\geq n_0$, there is a source element $v^{\dagger,n} \in H_n$
  with $\norm{Av^{\dagger,n}}_\infty \leq 1$, 
  $\abs{(Av^{\dagger,n})_i} = 1$ if and only if $i \in \supp u^\dagger$
  and with 
\be{epsn}
\varepsilon_{v^{\dagger,n}} \geq \frac12 \varepsilon_{v^\dagger}.
\ee
\end{lemma}

\begin{proof}
  Let $I$, $H_I$, $A_I$ be defined as in \eqref{Isupp}, \eqref{defHI}, \eqref{defAI}.
  As $\norm{P_nA^*e_i - A^* e_i}_H \to 0$ as $n \to \infty$
  for each $i \in I$ and as $I$ is finite,
  we also have convergence $P_nA_I^* \to A_I^*$ as $n \to \infty$ in the
  strong operator norm.
  Consequently, there is an $n_0 \in \NN$ and a $C > 0$ 
  such that for all $n\geq n_0$
  we have that $A_IP_nA_I^*$ is invertible with 
  $\norm{(A_IP_nA_I^*)^{-1}} \leq C$,
  where the latter norm is 
  the $\infty$-$1$-operator norm for linear mappings $\RR^I \to \RR^I$.
  %
  %
  %
  Hence, similarly to the proof of Lemma~\ref{lem:strict_source_elem}, 
  the problem of finding a solution to 
  \[
  v_I^n \in P_nH_I: \qquad
  \scp{v_I^n}{P_n A^*e_i} = \bar v_i^n \qquad
  \text{for} \ i \in I
  \]
  for $n \geq n_0$ given the coefficients $\bar v_i^n$ for $i \in I$,
  is well-posed and we have $\norm{v_I^n}_H
  \leq C \norm{A} \norm{\bar v^n}_\infty$ 
  for all $\bar v^n\in\RR^I$.

  By choosing $n_0$ possibly larger, we can achieve, 
  as $v^\dagger$ is chosen according to Lemma~\ref{lem:strict_source_elem},
  for $\rho$ as in \eqref{defrho}
  that $\norm{v^\dagger - P_nv^\dagger}_H \leq 
  \rho/(2\norm{A} (1 + C\norm{A}))$ for all $n \geq n_0$. 
  Then, we have for each $v_I^n \in P_nH_I$ 
  satisfying $\scp{v_I^n}{A^*e_i} = 
  \scp{v^\dagger - P_n v^\dagger}{A^*e_i}$,  $\forall i \in I$
  that $\norm{v_I^n}_H \leq C \norm{A}
  \norm{A(v^\dagger - P_n v^\dagger)}_\infty$.
  Thus, $v^{\dagger,n} = P_n v^\dagger + v_I^n \in H_n$ satisfies
  \[
  \scp{v^{\dagger,n}}{A^*e_i} = \scp{v^\dagger}{P_n A^* e_i} + 
  \scp{v^\dagger - P_n v^\dagger}{A^* e_i} = \scp{v^\dagger}{A^* e_i} 
  \in \sett{-1,1}
  \qquad \text{for} \ i \in I
  \]
  and, for $i \notin I$, we have by \eqref{defrho}
  \begin{align*}
    \abs{\scp{v^{\dagger,n}}{A^*e_i}} = &\abs{(A v^{\dagger,n})_i} \leq
    \abs{(A v^\dagger)_i} + \norm{A(P_n v^\dagger - v^\dagger)}_\infty + 
    \norm{Av^n_I}_\infty \\
    \leq & 1 - \rho + \frac{\rho \norm{A}}{2\norm{A}(1 + C \norm{A})}
    + \frac{\rho C \norm{A}^2}{2\norm{A}(1 + C \norm{A})} = 1 
      - \frac{\rho}2 < 1.
  \end{align*}
  Consequently, for $n \geq n_0$,
  $v^{\dagger,n}$ obeys $\norm{Av^{\dagger,n}}_\infty \leq 1$
  and 
  \[
  \scp{u^\dagger}{Av^{\dagger,n}} = 
  \sum_{i \in I} (Av^\dagger)_i u^\dagger_i 
  = \sum_{i \in I} \abs{u^\dagger_i} = \norm{u^\dagger}_1,
  \]
  meaning that $u^\dagger$ is a solution to~\eqref{eq:minprob} with $f^\delta=f$.
  %
  %
  %

  Next, suppose that $u^*$ is another solution of~\eqref{eq:minprob}  with $f^\delta=f$, which
  implies that $\scp{v^{\dagger,n}}{A^*u^*}=\scp{v^{\dagger,n}}{f}$ and $\norm{u^*}_1 
  = \norm{u^\dagger}_1$. Consequently, by construction of $v^{\dagger,n}$,
  \[
  \scp{u^*}{A v^{\dagger,n}} = \scp{v^{\dagger,n}}{f} = \norm{u^\dagger}_1 = \norm{u^*}_1.
  \]
  As $u^*$ has the representation
  $u^* = \norm{u^*}_1 \bigl( 
  \sum_{k \in I} \alpha_k \sigma_k e_k + \sum_{k \notin I} \alpha_k 
  \sigma_k e_k\bigr)$ with $\alpha_k \geq 0$,
  $\sigma_k \in \sett{-1,1}$, and $\sum_{k} \alpha_k = 1$, we deduce
  \begin{eqnarray*}
  \scp{u^*}{Av^{\dagger,n}} &=& \norm{u^*}_1 \scp{v^{\dagger,n}}{ 
  \sum_{k \in I} \alpha_k \sigma_k A^* e_k + \sum_{k \notin I} 
  \alpha_k \sigma_k A^* e_k} \\ &\leq&
  \norm{u^*}_1\Bigl( \sum_{k \in I} \alpha_k + \sum_{k \notin I} 
  \alpha_k \abs{\scp{v^{\dagger,n}}{A^*e_k}} \Bigr).
  \end{eqnarray*}
  As $\abs{\scp{v^{\dagger,n}}{A^*e_k}} < 1$ for each $k \notin I$, we conclude
  that 
  $\alpha_k = 0$ for each $k \notin I$ as otherwise, we would get
  the contradiction $\norm{u^*}_1 < \norm{u^*}_1$ from $\sum_{k\in\N}\alpha_k=1$. 
  Thus, identifying $\RR^I$ with the subspace of elements
  in $\lpspace{1}$ with support contained in $I$, we have
  $P_nA^* u^* = P_nA_I^* u^*$. Since $P_nA_I^*$ is invertible (see above)
  and $P_nA_I^* u^\dagger = P_n f = P_nA_I^* u^*$, it follows that
  $u^* = u^\dagger$, establishing uniqueness.

  Finally, $v^{\dagger,n}$ satisfies the stated properties by 
  construction. The construction also yields $\varepsilon_{v^{\dagger,n}}
  \geq \frac12 \varepsilon_{v^\dagger}$.
\end{proof}





\subsection{The structure of the approximations $u^n$}
Next, consider $f^\delta$ such that $\norm{f^\delta -  f}_H \leq \delta$.
To analyze the structure of a
solution $ u^n$, we consider the set
\[
K_n = \clconv(\set{\sigma P_n A^* e_i}{i \in \NN, \sigma \in \sett{-1,1}})
\subset H_n.
\]
Note that this set coincides with the closed unit ball associated
with the dual of the norm $v \mapsto \norm{Av}_\infty$ on $H_n$. This
set has an interior whose size can be estimated by $\frac{1}{\kappa_n}$, as
for $v,w \in H_n$ we have that $\norm{A w}_\infty \leq 1$ implies $\norm{w}_{H_n} \leq \kappa_n$, hence 
\[
\begin{aligned}
  \kappa_n\norm{v}_{H_n} =
  \sup_{\norm{w}_{H_n} \leq \kappa_n} \ \scp{v}{w} \leq 1 \ \ 
  & \Rightarrow \ \ \sup_{\norm{A v}_\infty \leq 1} \ \scp{v}{w} \leq 1,
\end{aligned}
\]
i.e.,
\be{BKn}
\mathcal{B}_{1/\kappa_n}(0)\subseteq K_n.
\ee
Furthermore, $(e_i)_i \wstarrightarrow 0$ as $i \to \infty$, so by
weak*-to-weak continuity, $A^*e_i \wrightarrow 0$ in $H$ and 
$P_n A^* e_i \to 0$ in $H_n$. 
By \eqref{BKn}, there exists $i_0\in\N$ such that for all $i\geq i_0$,
$\sigma P_n A^* e_i$ is not an extremal point of $K_n$.
Consequently, $K_n$ has only finitely many extremal points, i.e.,
is a convex polyhedron which is obviously also symmetric around $0$.
Furthermore, as $K_n$ has non-empty interior, 
the dual ball $K_n^* \subset H_n$ 
is also a symmetric convex polyhedron, 
i.e., possesses a finite extremal point set $K_{n,\ex}^*$. We see that
\[
K_n = \bigcap_{v \in K_{n,\ex}^*} \sett{\scp{v}{\placeholder} \leq 1}.
\]
Clearly, the extremal points of $K_n$ form  a 
symmetric
subset of the following set: $\set{\sigma P_n A^* e_i}{i \in \NN, \sigma \in \sett{-1,1}}$. 
These can be associated with the indices $i$ resulting in
\be{In}
I_n = \set{i \in \NN}{P_n A^* e_i \ \text{and} \ -P_n A^* e_i 
  \ \text{are extremal points of} \ K_n}
\ee
which is a finite set with at least $n$ elements (otherwise, $K_n$ would 
have empty interior).
By construction, each $v \in K_{n,\ex}^*$ obeys 
$\abs{\scp{v}{P_nA^* e_i}} \leq 1$ for each $i \in \NN$.

We observe that $K_{n,\ex}^*$ and
$I_n$ have a major influence on the structure of
solutions.
\begin{lemma}
  \label{lem:sparse_discrete_solutions}
  For $n$ fixed and 
  each $f^\delta \in H$, there is a sparsest solution $u^*_{\text{sparse}}$ 
  to~\eqref{eq:minprob} 
  with $(u^*_{\text{sparse}})_i \neq 0$ for $m$ distinct 
  elements $i$ in $I_n$, $(u^*_{\text{sparse}})_i = 0$ else
  and $m \leq n$.
  In particular, any other solution $u^{**}$ obeys 
  $\card\set{i \in \NN}{u^{**}_i \neq 0} \geq m$.
\end{lemma}

\begin{proof}
  Pick an $u^*$ satisfying \eqref{eq:minprob}. As for $P_n f^\delta = 0$, the statement
  is obviously true for $u^*_{\text{sparse}} = 0$, we may assume,
  without loss of generality, that $P_n f^\delta \neq 0$ and 
  $\norm{u^*}_1 > 0$.
  Now, there is a $v^* \in H_n$ with $\norm{A v^*}_\infty \leq 1$
  and $\scp{v^*}{P_n A^* e_i} = \sign(u^*_i)$ if $u_i^* \neq 0$. Moreover,
  for $\sigma_i^* = \sign(u^*_i)$ where $u^*_i \neq 0$, 
  we may write
  \[
  \frac{1}{\norm{u^*}_1} P_n f^\delta
  = \sum_{u_i^* \neq 0} \frac{\abs{u_i^*}}{\norm{u^*}_1} 
  \sigma_i^* P_n A^* e_i
  \qquad
  \Rightarrow
  \qquad
  \scp{v^*}{\frac{1}{\norm{u^*}_1} P_n f^\delta} = 1,
  \]
  meaning that $\norm{u^*}_1^{-1} 
  P_n f^\delta$ is a convex combination of 
  elements in $\sett{\scp{v^*}{\placeholder} = 1} \cap K_n$.
  Now, the extremal points of 
  $\sett{\scp{v^*}{\placeholder} = 1} \cap K_n$ 
  have to be extremal
  points of $K_n$: Otherwise, there is 
  an extremal point $w \in \sett{\scp{v^*}{\placeholder} = 1} \cap 
  K_n$ which has a representation
  $w = \alpha w_1 + (1-\alpha) w_2$ for $\alpha \in {]{0,1}[}$ and
  $w_1,w_2 \in K_n$, $w_1 \neq w_2$. By the extremal
  point property, not both
  $w_1$ and $w_2$ can be contained in 
  $\sett{\scp{v^*}{\placeholder} = 1} \cap K_n$. 
  Thus, $\scp{v^*}{w_1} \neq 1$ or $\scp{v^*}{w_2} \neq 1$. As 
  $\alpha \scp{v^*}{w_1} + (1-\alpha) \scp{v^*}{w_2} = 1$, either
  $\scp{v^*}{w_1} > 1$ or $\scp{v^*}{w_2} > 1$.
  This is, however, a contradiction to 
  $\norm{A v^*}_\infty \leq 1$ as $\scp{v^*}{\bar w} \leq 1$ for all
  $\bar w \in K_n$. Consequently, $w$ has to be an extremal point
  of $K_n$.

  By Carath\'eodory's theorem, we know that
  $\frac1{\norm{u^*}_1} P_n f^\delta$ is a convex combination of
  at most $n$ extremal points of $\sett{\scp{v^*}{\placeholder} = 1} \cap 
  K_n$, and hence, of at most $n$ extremal points of $K_n$, which
  implies
  \[
  P_n f^\delta = \norm{u^*}_1 \sum_{k = 1}^n \alpha_k 
  \sigma_k P_n A^* e_{i_k}
  \]
  for $\alpha_1,\ldots, \alpha_n \geq 0$, $\sum_{k=1}^n \alpha_k = 1$,
  $\sigma_k \in \sett{-1,1}$ and each $i_k \in I_n$.
  Thus, the minimum
  \[
  \min \ \Bigset{m \in \NN}{\exists i_1,\ldots,i_m \in I_n, 
    \ \alpha \in \RR^m, \norm{\alpha}_1 = \norm{u^*}_1,  \
  \sum_{k=1}^m \alpha_k P_n A^*e_{i_k} = P_nf^\delta}
  \]
  exists and is finite.
  It is then clear that for an $\alpha \in \RR^m$ associated with
  an optimal $m$ we have $\alpha_k \neq 0$ for all $k = 1,\ldots,m$.
  By construction, 
  $u^*_{\text{sparse}} = \sum_{k=1}^m \alpha_k e_{i_k}$ with $m$ admitting the
  above minimum and $\alpha \in \RR^m$ according to the definition,
  is a sparsest solution.
\end{proof}

\begin{remark}
  From 
  Lemma~\ref{lem:discrete_solution_exact_data} is follows that
  there is a $n_0$ such that
  for $n \geq n_0$, $u^\dagger$ is the sparsest solution with data $f$.
\end{remark}



\begin{lemma}
  Each $u^*$ solution of~\eqref{eq:minprob} can be represented
  as a finite convex combination of solutions with minimal support
  in
  the following sense: A solution $u^*$ has minimal  support
  if 
  for any other solution $u^{**}$ with $\supp u^{**} \subset \supp u^*$
  it follows $u^{**} = u^*$.
\end{lemma}

\begin{proof}
  Obviously, the set of solutions $S$ is a non-empty, convex and bounded
  subset of $\lpspace{1}$. It is moreover contained in a
  finite-dimensional subspace of $\lpspace{1}$. To see this, let
  $u^*$ be a solution of~\eqref{eq:minprob} and $v^* \in H_n$ 
  such that
  $\norm{Av^*}_\infty \leq 1$ and $\scp{Av^*}{u^*} = \norm{u^*}_1$.
  Then, as $A$ maps into $\ccspace$, there is an $i_0$ such that for
  all $i \geq i_0$, $\abs{Av_i^*} < 1$ and consequently, $u^*_i = 0$.
  For any other solution $u^{**}$ we get
  \[
  \scp{Av^*}{u^*} 
  = \scp{v^*}{P_nA^*u^*} = 
  \norm{u^*}_1 = \norm{u^{**}}_1 = 
  \scp{v^*}{P_nA^*u^{**}} = \scp{Av^n}{u^{**}}.
  \]
  Consequently, $u^{**}_i = 0$ for all $i \geq i_0$. Thus, $S$ is 
  contained in a finite-dimensional subspace of $\lpspace{1}$.

  Being a non-empty, convex and compact subset of a finite-dimensional
  space, each element in $S$ can be represented by a finite convex
  combination of its extremal points. Let us verify that the
  extremal points satisfy the stated minimality property. For
  that purpose, let $u^*$ be an extremal point of $S$ with
  $u^* = \sum_{k=1}^\kappa u_k^* e_{i_k}$ for $i_1,\ldots,i_\kappa$ distinct
  and each $u_k^* \neq 0$.
  Now, either the collection
  $\sett{P_nA^*e_{i_1}, \ldots, P_nA^*e_{i_\kappa}}$ is linearly independent
  or not. However, the case that these vectors are linearly dependent
  can be excluded as follows. 
  Choose a $u = \sum_{k=1}^\kappa u_k e_{i_k} \neq 0$
  such that $P_nA^*u = 0$. Then, for $\varepsilon > 0$ small enough
  we can achieve that $u^\varepsilon = u^* + \varepsilon u$ as well as
  $u^* - \varepsilon u$ are still solutions: Indeed, 
  $P_nA^*u^{\pm\varepsilon} = P_nA^* u^*$ is satisfied 
  for each $\varepsilon > 0$.
  Additionally, for $\sign(u^{\pm\varepsilon}_k) = \sign(u^*_k)$ for each
  $k$ (which can be achieved for  $\varepsilon$ small enough) 
  we have $\scp{Av^*}{u^{\pm\varepsilon}} = \norm{u^{\pm\varepsilon}}_1$,
  meaning that $u^{\pm\varepsilon}$ is a solution. However,
  $u^* = \frac12 u^\varepsilon + \frac12 u^{-\varepsilon}$ and
  $u^\varepsilon \neq u^{-\varepsilon}$, so $u^*$ cannot be an extremal
  point.
  Consequently, 
  $\sett{P_nA^*e_{i_1}, \ldots, P_nA^*e_{i_\kappa}}$ are linearly independent.
  Thus, if $u^{**}$ is a solution with 
  $\supp u^{**} \subset \supp u^*$, we have the representation 
  $u^{**} = \sum_{k=1}^\kappa u_k^{**} e_{i_k}$. However, $P_nA^*u^* = P_nA^*u^{**}$
  and by injectivity of $(u_1,\ldots,u_{\kappa}) \mapsto
  P_n A^* \sum_{k=1}^\kappa u_k e_{i_k}$, $u^*_k = u^{**}_k$ for all
  $k = 1,\ldots, \kappa$, i.e., $u^* = u^{**}$.
\end{proof}



\subsection{Error estimates}

For proving an $O(\delta)$ convergence rate, we will  
assume that $n$ is that large to
ensure $u^\dagger$ is the sparsest 
solution to~\eqref{eq:minprob} for the data
$ f$. This means in particular that
$\supp u^\dagger \subset I_n$ (cf. \eqref{In}) with
$\set{P_nA^*e_i}{i \in \supp u^\dagger}$ consisting
of linearly independent vectors.
Let $f^\delta \in H$ be such that \eqref{delta} holds for $\delta > 0$. 
Denote by $u^{n}$ a 
solution of~\eqref{eq:minprob} which has to be sparse as for the 
corresponding source element $v^n \in H_n$, we have $Av^n \in c_0$.
Lemma~\ref{lem:sparse_discrete_solutions} states, however, that without
loss of generality, $\supp u^n$ possesses at most $n$ elements.




\begin{theorem}\label{Odelta}
  There exists a $C > 0$ and an $n_0$ such that for
  $n \geq n_0$ and $f^\delta \in H$ with
  $\norm{f^\delta -  f}_1 \leq \delta$,
  for any solution $u^{n}$ of~\eqref{eq:minprob}
  with data $f^\delta$ it holds that
  \[
  \norm{u^{n} - u^\dagger}_1\leq C \delta\kappa_n.
  \]
\end{theorem}

\begin{proof}
  Choosing $n_0$ according to Lemma~\ref{lem:discrete_solution_exact_data} and, for $n\geq n_0$,  
  denoting by $v^{\dagger,n} \in H_n$
  the source element associated
  with the solution $u^\dagger$ 
  with data $ f$
  according to Lemma~\ref{lem:discrete_solution_exact_data}, and by
  $u^{n}$ a sparse solution of~\eqref{eq:minprob} according to 
Lemma~\ref{lem:sparse_discrete_solutions} with data
  $f^\delta$ with source element $v^{n} \in H_n$,
  we have,
  according to Proposition~\ref{prop:stabest} 
  for the Bregman distance $D$ associated with the
  subgradient element element $Av^{\dagger,n}$ and $D^{\sym}$ the
  symmetric Bregman distance associated with the 
  subgradient elements $v^{\dagger,n}$ and $v^{n}$, respectively,
  that
  \[
  D(u^{n}, u^\dagger) \leq D^{\sym}(u^{n}, u^\dagger)
  \leq 2\kappa_n\norm{f^{\delta} -  f}_H
  \leq 2\delta\kappa_n.
  \]
  On the other hand, by definition and with the operator $P$ defined by $(P u^{n})_i = 
  u^{n}_i$ if $i \notin \supp u^\dagger$ and $0$ otherwise,
  we have
  \[
  D(u^{n}, u^\dagger) \geq \sum_{i \notin \supp u^\dagger}
  \bigr(1 - \abs{(Av^{\dagger,n})_i} \bigr) \abs{u^{n}}
  \geq \varepsilon_{v^{\dagger,n}} \norm{
    P u^{n}}_1
  \geq \frac{\varepsilon_{v^\dagger}}{2} 
  \norm{
    P u^{n}}_1.
  \]
  In total, with $Pu^\dagger=0$, one concludes 
\[\norm{P(u^{n} - u^\dagger)}_1=\norm{Pu^{n}}_1
  \leq \frac{4\delta\kappa_n}{\varepsilon_{v^\dagger}}\,.
\]
  
  Furthermore, 
for $Q=\mbox{id}-P$ we see by $Qu^\dagger=u^\dagger$ that
  \begin{align*}
  &\norm{P_nA^*Q(u^{n} - u^\dagger)}_H 
  = \norm{P_n(f^\delta -  f) - P_nA^*Pu^{n}}_H\\
  &\leq \delta + \norm{A} \frac{4\delta\kappa_n}{\varepsilon_{v^\dagger} }
  \leq \frac{\varepsilon_{v^\dagger}\kappa_n^{-1} + 4\norm{A}}
  {\varepsilon_{v^\dagger}} \delta\kappa_n
    \leq \frac{(\varepsilon_{v^\dagger} + 4) \norm{A}}{\varepsilon_{v^\dagger}}
    \delta \kappa_n
  \end{align*}
  where $\kappa_n^{-1} \leq \norm{A}$ follows from the definition of
  $\kappa_n$ in~\eqref{kappa}.
  Now, for 
$I$ as in \eqref{Isupp}, $A_I$ as in \eqref{defAI},
we can choose
  $C_I > 0$ such that $\norm{(A_IP_nA_I^*)^{-1}} \leq C_I$ for all
  $n \geq n_0$.
  Consequently,
  \begin{align*}
  \norm{Q(u^{n} - u^\dagger)}_1 &= 
  \norm{(A_IP_nA_I^*)^{-1}A_IP_nA^*Q(u^{n} - u^\dagger)}_1\\
  &\leq \frac{C_I (\varepsilon_{v^\dagger} + 4)\norm{A}^2\delta\kappa_n}
  {\varepsilon_{v^\dagger}} .
  \end{align*}
  In total, we have
  \[
  \norm{u^{n} - u^\dagger}_1 \leq 
  \frac{4 + C_I (\varepsilon_{v^\dagger} + 4)\norm{A}^2}
  {\varepsilon_{v^\dagger}}\delta\kappa_n
  = C \delta\kappa_n.
  \]
  which is the desired statement.
  %
  %
\end{proof}

\medskip

With this result, the choice $n=n_0$ gives an $O(\delta)$ estimate of the error $\norm{u^{n} - u^\dagger}_1$. However, $n_0$ is not known a priori. We now show that under certain assumptions it can be replaced by $n=n_{DP}(\delta)$ according to the discrepancy principle, again for general solutions of~\eqref{eq:minprob} and without needing $v^{\dagger,n}$ from Lemma~\ref{lem:discrete_solution_exact_data}, but just relying on the source condition Assumption~\ref{ass:sc} and the specially constructed source element $v^\dagger=v\ddagger$ from Lemma~\ref{lem:strict_source_elem}. Note however, that this will --- besides requiring additional assumptions such as \eqref{C1} --- typically also not lead to the ideal rate $O(\delta)$.
\begin{theorem}
Let \eqref{C1} hold. Then there exists $C>0$ such that for $f^\delta \in H$ satisfying \eqref{delta} and any solution $u_{n_{DP}(\delta)}$ of~\eqref{eq:minprob} with $n=n_{DP}(\delta)$ according to \eqref{discrprinc}, it holds that
  \be{ratedp1}
  \norm{u^{n_{DP}(\delta)} - u^\dagger}_1\leq C \delta\kappa_{n_{DP}(\delta)}.
  \ee
If additionally, for some $C_3>0$ and all $n\in\N$, 
\begin{equation}\label{C3}
\kappa_n\norm{(\mbox{id}-P_{n-1})A^*}\leq C_3
\end{equation}
and, for some index function $\Psi$, (i.e., a strictly monotone function satisfying $\Psi\to0$ as $t\to0$)
\be{Psi}
\norm{A^*(u^{\dagger,n}-u^\dagger)}\leq\Psi\left(\frac{1}{\kappa_{n+1}}\right)
\ee
holds, then 
  \be{ratedp2}
  \norm{u^{n_{DP}(\delta)} - u^\dagger}_1= O\left(\frac{\delta}{\Phi^{-1}(\frac{\delta}{\tilde{C}})}\right),
  \ee
\end{theorem}
where $\Phi(\lambda)=\lambda\Psi(\lambda)$ and $\tilde{C}>0$ is a constant independent of $\delta$.
\begin{proof}
With $n=n_{DP}(\delta)$ and using \eqref{in} we get that 
\begin{align*}
D(u^n, u^\dagger) &=\norm{u^n}_1- \langle A v^\dagger,u^n \rangle\\
&=\norm{u^n}_1- \langle Av^\dagger, u^\dagger\rangle +\langle v^\dagger, f-f^\delta\rangle
-\langle v^\dagger, A^*u^n-f^\delta\rangle\\
&\leq\Bigl(\kappa_n+(\tau+1)\norm{v^\dagger}\Bigr)\delta\,.
\end{align*}
Similarly to the proof of Theorem~\ref{Odelta} above but a bit simpler (since we do not need the source elements $v^{\dagger,n}$ here)  we get
  \[
  D(u^n, u^\dagger) \geq \sum_{i \notin \supp u^\dagger}
  \bigr(1 - \abs{(Av^\dagger)_i} \bigr) \abs{u^n}
  \geq \varepsilon_{v^\dagger} 
  \norm{Pu^n}_1,
  \]
hence 
\be{Pun}
\norm{Pu^n}_1=\norm{P(u^n-u^\dagger)}_1\leq\frac{\kappa_n+(\tau+1)\norm{v^\dagger}}{\varepsilon_{v^\dagger} }
\delta.
\ee
On the other hand,
\begin{align*}
&\norm{A^*(\mbox{id}-P)(u^n-u^\dagger)}\leq
\norm{A^*(u^n-u^\dagger)} + \norm{A}\norm{P(u^n-u^\dagger)}_1\\
&=\norm{A^*u^n-f} + \norm{A}\norm{P(u^n-u^\dagger)}_1
\leq \left(\tau+1+ \norm{A}\frac{\kappa_n+(\tau+1)\norm{v^\dagger}}{\varepsilon_{v^\dagger}}\right)\delta\,,
\end{align*}
hence by boundedness of $(A^*(\mbox{id}-P))^\dagger:=((A^*(\mbox{id}-P))^*A^*(\mbox{id}-P))^{-1}(A^*(\mbox{id}-P))^*$ (see the proof of Lemma \ref{lem:strict_source_elem}) and the fact that $(A^*(\mbox{id}-P))^\dagger A^*(\mbox{id}-P)=(\mbox{id}-P)$
we get 
\be{ImPun}
\norm{(\mbox{id}-P)(u^n-u^\dagger)}_1\leq C \left(\tau+1+ \norm{A}\frac{\kappa_n+(\tau+1)\norm{v^\dagger}}{\varepsilon_{v^\dagger}}\right)\delta\,.
\ee
Combining \eqref{Pun}, \eqref{ImPun} yields \eqref{ratedp1}.
\\
To obtain a convergence rate with respect to $\delta$ it is essential to estimate $\hat{\gamma}_n$ from above and below. For the former purpose, we proceed as above, but this time for the noise free discrete approximation and using the infinite dimensional source element $v^\dagger$. Namely, using the fact that by minimality $\norm{u^{\dagger,n}}_1\leq\norm{u^\dagger}_1$,
\begin{align*}
D(u^{\dagger,n}, u^\dagger) &=\norm{u^{\dagger,n}}_1- \langle A v^\dagger,u^{\dagger,n}\rangle
  \leq \langle v^\dagger, A^*(u^{\dagger}- u^{\dagger,n})\rangle,  \\
D(u^{\dagger,n}, u^\dagger) &\geq \sum_{i \notin \supp u^\dagger}
\bigr(1 - \abs{(Av^\dagger)_i} \bigr) \abs{u^{\dagger,n}}
\geq \varepsilon_{v^\dagger} 
\norm{Pu^{\dagger,n}}_1,
\end{align*}
one has
\[
\norm{Pu^{\dagger,n}}_1=\norm{P(u^{\dagger,n}-u^\dagger)}_1\leq\frac{\norm{v^\dagger}}{\varepsilon_{v^\dagger} }
\norm{A^*(u^{\dagger,n}- u^\dagger)}
\]
as well as
\begin{align*}
&\norm{(\mbox{id}-P)(u^{\dagger,n}-u^\dagger)}_1\leq C \norm{A^*(\mbox{id}-P)(u^{\dagger,n}-u^\dagger)}\\
&\leq
C\norm{A^*(u^{\dagger,n}-u^\dagger)} + C\norm{A}\norm{P(u^{\dagger,n}-u^\dagger)}_1
\leq C\left(1+\frac{\norm{A}\norm{v^\dagger}}{\varepsilon_{v^\dagger} }\right)
\norm{A^*(u^{\dagger,n}- u^\dagger)}.
\end{align*}
Altogether, one obtains
\[
\norm{u^{\dagger,n}-u^\dagger}_1\leq 
\left(C+\left(1+C\norm{A}\right)\frac{\norm{v^\dagger}}{\varepsilon_{v^\dagger} }\right)
\norm{A^*(u^{\dagger,n}- u^\dagger)}\,.
\]
Inserting this into the definition of $\hat{\gamma}_n$ \eqref{gamman} and using \eqref{C3}, \eqref{Psi} yields
\begin{align*}
\hat{\gamma}_n&=\norm{(\mbox{id}-P_n)A^*(u^{\dagger,n}-u^\dagger)}\leq \norm{(\mbox{id}-P_n)A^*}\norm{u^{\dagger,n}- u^\dagger}\\
&\leq\frac{C_3}{\kappa_{n+1}} \left(C+\left(1+C\norm{A}\right)\frac{\norm{v^\dagger}}{\varepsilon_{v^\dagger} }\right) \norm{A^*(u^{\dagger,n}- u^\dagger)}\leq \bar{C}\frac{\Psi(\frac{1}{\kappa_{n+1}})}{\kappa_{n+1}}
\end{align*}
Thus from \eqref{estNDPm1} we conclude
\[
\delta \leq\tilde{C} \Phi\left(\frac{1}{\kappa_{n_{DP}(\delta)}}\right),
\]
i.e., 
\[
\frac{1}{\Phi^{-1}\left(\frac{\delta}{\tilde{C}}\right)} \geq \kappa_{n_{DP}(\delta)}\,.
\]
Inserting this into \eqref{ratedp1} yields \eqref{ratedp2}.
\end{proof}

\section{Particularities of the method}

The aim of this section is to deeper understand the effect of the least error method  as a discretization method for relevant bases.

\begin{enumerate}
\item \textit{The case of the singular basis}\\
Let, for $\mH$ a Hilbert space, the linear 
operator $\mA: H \to \mH$ be compact
and let $(\sigma_n,\hat v^n, \hat u^n)$ be a singular basis of the
compact operator $\mA$. Here,
$(\sigma_n)_n$ stands for the non-increasingly
ordered sequences of positive singular values converging to zero as 
$n \to \infty$ and $(\hat v^n)_n$, $(\hat u^n)_n$ are orthonormal systems
in $H$ and $\mH$, respectively. Then,
\[
\mA^* u = \sum_{i} \sigma_i \scp{u}{\hat u^i} \hat v^i
\]
and with the basis operator $T: \mH \to c_0$, $(Tu)_i = \scp{u}{\hat u^i}$, 
the least error framework for the solution of
$\mA^* u = f^\delta$ 
may be applied to $A = T \mA$. With the
choice $H_n = \mbox{span}(\hat v^1, \ldots, \hat v^n)$, this results in
\be{le_singular}
u^n \in\mbox{argmin}\{ \norm{Tu}_1 \, : \, \forall i=1,2,...,n \, : \ \sigma_i\langle u, \hat u^i\rangle=\dup{ \hat v^i,\fd}\}.
\ee
%
%
%
%
\item \textit{The case of the canonical basis}\\
  Consider the canonical basis $(e_n)_n$ in $H=\ell^2$ and
  $H_n = \text{span}(e_1,\ldots,e_n)$.
  Thus,   one can re-formulate \eqref{le}  as
  \be{le_canonical}
  u^n \in\mbox{argmin}\{ \norm{u}_1 \, : \, \forall i=1,2,...,n \, : \ (A^*u)_i 
  =\fd_i\}.
\ee
If one considers the denoising problem, then the operator $A^*$ in this case is just the embedding operator from $\ell^1$ into $\ell^2$ and 
\be{le_denoise}
u^n \in\mbox{argmin}\{ \sum_{i>n}|u_i| \, : \, \forall i=1,2,...,n \, : \ u_i =\fd_i\},
\ee
where the first $n$ components of the regularized solution $u^n$ coincide with the first $n$ components of the noisy data. Since the minimizer of the above problem is attained when $u_i =0$, for all $i>n$, one obtains $$u^n=(\fd_1,\fd_2,...,\fd_n,0,0,...).$$
\end{enumerate}

%
%
%
%
%

\section{Conclusions and Remarks}
In this paper we have provided a stability and convergence analyis for the least error method with $\ell^1$ as a preimage space. We have proven convergence rates under a source condition, even with respect to the norm topology, and shown that the method indeed leads to sparse approximations. The analysis includes detailed investigations on the source elements, which are crucial for stability estimates leading to ideal $O(\delta)$ convergence rates.

Future research will be concerned with an efficient implementation of the method as well as numerical tests.
Moreover we are working on an extension of the approach to a sparsity enhancing method in a function space setting with spaces of Radon measures in place of $\ell^1$.

\section*{Acknowledgments}
The second author gratefully acknowledges financial support by the Austrian Science Fund FWF under grant I 2271.
The second and third author are supported by the Karl Popper Kolleg ``Modeling --  Simulation -- Optimization'' funded by the Alpen-Adria-Universit\"at Klagenfurt and by the Carinthian Economic Promotion Fund (KWF).
The Institute of Mathematics and
Scientific Computing of the University of Graz, with which the first author
is affiliated, is a member of NAWI Graz
(\texttt{http://www.nawigraz.at}).

\bigskip
\bibliographystyle{siam}
\bibliography{least_error_ell1}

\end{document}